\documentclass[psamsfonts,11pt,reqno]{amsart}
\usepackage[english]{babel}
\usepackage{amsmath,amsfonts,amssymb,amsthm,microtype,graphicx,mathtools,pgfplots,enumitem}
\usepackage[mathscr]{eucal}
\usepackage{footnote}
\usepackage{sepfootnotes}
\usepackage{blindtext}
\usepackage[hyphens]{url}
\usepackage{float}
\DeclarePairedDelimiter\abs{\lvert}{\rvert}
\DeclarePairedDelimiter\norm{\lVert}{\rVert}

\theoremstyle{definition}
\newtheorem{theorem}{Theorem}[section]

\newtheorem{lemma}[theorem]{Lemma}

\newtheorem{proposition}[theorem]{Proposition}
\newtheorem{remark}[theorem]{Remark}
\numberwithin{equation}{section}

\usepackage{hyperref}
\begin{document}
	\title[OPTIMAL REGULARITY BY AN EPIPERIMETRIC INEQUALITY]{OPTIMAL REGULARITY OF THE THIN OBSTACLE PROBLEM BY AN EPIPERIMETRIC INEQUALITY}
	\author{MATTEO CARDUCCI}
	\address{Department of Mathematics "Guido Castelnuovo", University of Rome "La Sapienza", Piazzale Aldo Moro 5, 00185, Roma.}
	\email{\href{mailto:carducci.1868932@studenti.uniroma1.it}{carducci.1868932@studenti.uniroma1.it}}
	\begin{abstract}
		The key point to prove the optimal $C^{1,\frac12}$ regularity of the thin obstacle problem is that the frequency at a point of the free boundary $x_0\in\Gamma(u)$, say $N^{x_0}(0^+,u)$, satisfies the lower bound $N^{x_0}(0^+,u)\ge\frac32$.
		
		In this paper we show an alternative method to prove this estimate, using an epiperimetric inequality for negative energies $W_\frac32$. It allows to say that there are not $\lambda-$homogeneous global solutions with $\lambda\in (1,\frac32)$, and
		by this this frequancy gap, we obtain the desired lower bound, thus a new self contained proof of the optimal regularity.
	\end{abstract}
	\maketitle
	\section{Introduction}
	\subsection{The thin obstacle problem.} Let $B_1=\{ x\in \mathbb{R}^{n+1}: \abs x <1\}\subset \mathbb{R}^{n+1}$ and $B'_1:=B_1\cap \{x_{n+1}=0\}$, we consider a solution of the thin obstacle problem, with obstacle $0$, i.e. a function $u\in H^1(B_1)$ which is a minimizer of the Dirichlet energy $$\mathcal{E}(v)=\int_{B_1} \lvert\nabla v\rvert ^2\,dx$$ among the admissible functions $$\mathcal{K}_{g}=\{v\in H^1(B_1):  v\ge0 \mbox{ on } B'_1,\ v=g \mbox{ on } \partial B_1,\ v(x',x_{n+1})=v(x',-x_{n+1}) \},$$ where $g=u|_{\partial B_1}$ in trace sense and $x'=(x_1,\ldots,x_{n})$ denotes the first $n$ coordinates.
	
	We denote with $\Lambda(u)=\{u=0\}\cap B'_1$ the contact set and $\Gamma(u)=\partial'\Lambda(u)$ the free boundary, where $\partial'$  is the boundary in $B'_1$ topology.
	
	The solution $u$ satisfies the corresponding Euler-Lagrange equations \begin{equation*} 
		\begin{cases}
			\Delta u=2\partial_{n+1}u\mathcal{H}^n|_{\Lambda(u)}\le0& \mbox{in } B_1\\
			u\ge0 \quad \partial_{n+1}u\le0 & \mbox{on } B'_1\\
			u\partial_{n+1}u=0 & \mbox{on } B'_1,
		\end{cases}
	\end{equation*}
	in particular $\Delta u=0$ in $B_1\setminus \Lambda(u)$. We recall the third line of the Euler-Lagrange equations as Signorini ambiguous conditions.
	\subsection{State of art} The existence and uniqueness of the solution was established in \cite{ls67} by Lions and Stampacchia, while, Caffarelli proved the $C^{1,\alpha}$ regularity in the either side of the obstacle, in \cite{caf79}. The optimal regularity, which is $C^{1,\frac12}$ in the either side of the obstacle, was shown after 25 years, in \cite{ac04} by Athanasopoulos and Caffarelli. 
	
	In \cite{acs08}, Athanasopoulos, Caffarelli and Salsa started the study of the free boundary, using the monotonicity formula of Almgren's frequency function $$N^{x_0}(r,u):=\frac{r\int_{B_r(x_0)} \lvert \nabla u \rvert^2\,dx}{\int_{\partial B_r(x_0)} u ^2\,d\mathcal{H}^n},$$ for $x_0\in \Gamma(u)$, originally introduced in \cite{alm00}.
	
	This formula allows to define the frequency of a point $x_0\in\Gamma(u)$ as $$N^{x_0}(0^+,u):=\lim_{r\to0^+}N^{x_0}(r,u)$$
	and to decompose the free boundary $\Gamma(u)$ as
	$$\Gamma_\lambda(u):=\{x_0\in \Gamma(u): N^{x_0}(0^+,u)=\lambda\}.$$ 
	It is well known (see \cite{acs08}) that \begin{equation*}\Gamma(u)=\Gamma_\frac32(u)\cup\bigcup_{\lambda\ge2}\Gamma_\lambda(u),
	\end{equation*} 
	but in this paper we will give a new proof to this fact, using an epiperimetric inequality (actually only $\Gamma(u)=\bigcup_{\lambda\ge\frac32}\Gamma_\lambda(u)$). 
	
	An approach to see the regularity and structure of the free boundary $\Gamma(u)$ uses the epiperimetric inequalities. Introduced by \cite{wei99} for the classical obstacle problem, the Weiss' energy of a function $u\in H^1(B_1)$ was generalized to the thin obstacle problem as \begin{equation*}\label{weiss}W_\lambda(r,u):=\frac{1}{r^{n-1+2\lambda}}\int_{B_r}\lvert\nabla u\rvert^2\,dx-\frac{\lambda}{r^{n+2\lambda}}\int_{\partial B_r}u^2\,d\mathcal{H}^n,
	\end{equation*} and we denote with $W_\lambda(u)=W_\lambda(1,u)$.
	
	Focardi and Spadaro in \cite{fs16} proved, by contradiction, an epiperimetric inequality for $\lambda=\frac32$ for positive energies (see Note \ref{posneg}). It say that: if $z\in\mathcal{K}_c$ is $\frac32-$homogeneous, then there is $ \zeta\in\mathcal{K}_c$ such that $$W_{\frac{3}{2}}(\zeta)\le (1-\kappa)W_{\frac{3}{2}}(z).$$ 
	
	A direct proof of the epiperimetric inequality for $W_{\frac32}$ was proved by Colombo, Spolaor and Velichkov in \cite{csv17}, where it was established the best possible constant $\kappa=\frac{1}{2n+5}$. Moreover, in \cite{csv17}, was proved 2 epiperimetric inequalities for $\lambda=2m$: a log-epiperimetric inequality for positive energies $W_{2m}$, and an epiperimetric inequality for negative energies $W_{2m}$ (see Note \ref{posneg}). They say that: if $z\in \mathcal{K}_c$ is $2m-$homogeneous, then there is $\zeta\in\mathcal{K}_c$ such that $$W_{2m}(\zeta)\le W_{2m}(z)(1-\varepsilon\lvert W_{2m}(z)\rvert^\gamma),$$ and there is $\zeta\in\mathcal{K}_c$ such that $$W_{{2m}}(\zeta)\le (1+\varepsilon)W_{2m}(z).$$
	
	Another epiperimetric inequality, both positive and negative energies, for the classical obstacle problem was proved in \cite{esv23} by Edelen, Spolaor and Velichkov.
	
	These epiperimetric inequalities for the thin obstacle problem, allow us to prove the previous results of regularity and structure of the free boundary, established in \cite{acs08} and \cite{gp09}, i.e. $\Gamma_\frac32(u)$ is a locally manifold $C^{1,\alpha}$ of dimension $n-1$ and $\bigcup_{m\in\mathbb{N}}\Gamma_{2m}(u)$ is contained in a countable union of manifolds $C^{1}$ (improved to $C^{1,\log}$ using the log-epiperimetric inequality).
	
	Another question about the thin obstacle problem is the admissible values of the frequency.
	
	We will say $\lambda\ge0$ is an admissible frequency if there is a non-trivial solution $u\in H^1(B_1)$, which is $\lambda-$homogeneous. In particular, if there is $u\in H^1(B_1)$ solution and $\Gamma_\lambda(u)\not=\emptyset$, then $\lambda$ is an admissible frequency, by Proposition \ref{blow}.
	
	One way to prove some estimates for the admissible frequencies is uses the epiperimetric inequalities. In fact, they allows to show a frequency gap for the $\lambda-$homogeneous solutions: forward using the epiperimetric inequalities for positive energies, and backward using the epiperimetric inequalities for negative energies (see Note \ref{posneg}), as shown in Proposition \ref{gap}.
	
	For instance, the constant $\kappa=\frac{1}{2n+5}$ for $W_\frac32$ is the best possible constant, since allows to prove a frequency gap in $(\frac32,2)$ (notice that there is a solution 2-homogeneous). On the other hand, the epiperimetric inequalities for $W_{2m}$ allows to prove that $\Gamma_\lambda(u)=\emptyset$ for $\lambda\in (2m-c_m^-,2m+c_m^+)$, for some $c_m^+$ and $c_m^-$ small explicit constants.
	\subsection{Main results.} 
	We give an alternative proof to the optimal $C^{1,\frac12}$ regularity of the solution of the thin obstacle problem, which was first obtained in the paper of Athanasopoulos and Caffarelli \cite{ac04}. 
	
	The key point is to prove the lower bound \begin{equation}\label{f}N^{x_0}(0^+,u)\ge\frac32 \quad \forall x_0\in\Gamma(u),
	\end{equation}
	from which the optimal regularity follows by a standard argument in the free boundary problems.
	
	In the literature, there are several proofs of the estimate \eqref{f}. For instance, we can use the Alt-Caffarelli-Friedman's monotonicity formula (see \cite{psu12}), or the semiconvexity of the solution (see \cite{survey}), or the Federer's dimension reduction principle (see \cite{fs}). 
	
	For the proof of the lower bound \eqref{f}, we prove a new epiperimetric inequality for negative energies\footnote{The word "negative energies" (respectively "positive energies") emphatizes that is relevant only for $W_\frac32(z)<0$ (respectively $W_\frac32(z)>0$), since in the other case we can choose $\zeta\equiv z$.\label{posneg}} $W_\frac32$, that is the following theorem.
	\begin{theorem}\label{thm1} Let $z=r^\frac32c(\theta)\in\mathcal{K}_c$, the $\frac32-$homogeneous extension in $\mathbb{R}^{n+1}$ of $c\in H^1(\partial B_1)$, then there is $\zeta\in\mathcal{K}_c$ such that $$W_{\frac{3}{2}}(\zeta)\le (1+\varepsilon)W_{\frac{3}{2}}(z),$$ with $\varepsilon=\frac{1}{2n+3},$ dimensional constant.
	\end{theorem}
	
	Using this epiperimetric inequality, we can deduce a backward frequency gap of $\frac32$ until $1$ and the estimate \eqref{f}, as in the next proposition. Moreover, the frequency gap shows that $\varepsilon =\frac{1}{2n+3}$ is the best possible constant, since there is a solution $1-$homogeneous, that is $-\lvert x_{n+1}\rvert$.
	\begin{proposition}\label{gap}
		Let $z\in H^1_{loc}(\mathbb{R}^{n+1})$ a $\lambda-$homogeneous global solution of the thin obstacle problem, then $\lambda\not\in (1,\frac32)$.
		
		In particular if $u\in H^1(B_1)$ is a solution, then $$N^{x_0}(0^+,u)\ge\frac32 \quad \forall x_0\in\Gamma(u),$$ i.e. \eqref{f} holds.
	\end{proposition}
	By the first statement of Proposition \ref{gap}, we deduce the estimate \eqref{f}. In fact it is well known that if $x_0\in\Gamma(u)$, then $\lambda=N^{x_0}(0^+,u)>1$ and the blow-up $u_0$ is $\lambda-$homogeneous (see Proposition \ref{blow}). Therefore, since there are not $\lambda-$homogeneous solutions with $\lambda \in(1,\frac32)$, we deduce $\lambda\ge\frac32$, which is \eqref{f}.
	
	For the sake of completeness, in Theorem \ref{c112} we use the lower bound \eqref{f} to conclude a $C^{1,\frac12}$ estimate, as desired.
	
	Finally, we give an alternative proof of the characterization of the $\frac32-$homogeneous solutions.
	
	\begin{proposition}\label{uniq} Let $z\in H^1_{loc}(\mathbb{R}^{n+1})$ a $\frac32-$homogeneous global solution of the thin obstacle problem, then $z=Ch_e(x',x_{n+1}),$ for some $C\ge0$, where 
	\begin{equation}\label{he}h_e(x',x_{n+1})=\mbox{Re}((x'\cdot e+i\lvert x_{n+1}\rvert)^\frac{3}{2}),
	\end{equation}
	with $e\in\partial B'_1$.
	\end{proposition}

	For alternative proofs of this results see \cite{acs08} or \cite{psu12}.
	
	\subsection*{Acknowledgment.} 
	I would like to thank Bozhidar Velichkov and Luca Spolaor for the invitation to the workshop "Regularity Theory for Free Boundary and Geometric Variational Problems III" in Levico Terme, where I started and did much of the work that led to this paper. 
	
	In particular I thank Bozhidar Velichkov for the useful discussions.

	I would like also thank Emanuele Spadaro for introducing me to the thin obstacle problem, which I was able to study in my master's degree thesis.
	\section{Preliminaries}
	Let's start with a short reminder of some of the most important results that we will use later.
	\subsection{Non-optimal $C^{1,\alpha}$ regularity}
	The purpose of this paper, as well as the prove of a new epiperimetric inequality, is to prove the optimal $C^{1,\frac12}$ regularity. The $C^{1,\alpha}$ regularity for some (non-optimal) $\alpha>0$ can be obtained in several different ways, by some standard elliptic PDE arguments. Precisely, we have the following theorem. 
	
	\begin{theorem} \label{c11a}($C^{1,\alpha}$ estimate)
		Let $u\in H^1(B_1)$ a solution of the thin ostacle problem, then $u\in C^{1,\alpha}_{loc}(B_1^+\cup B'_1)$ with the estimate $$\norm u_{C^{1,\alpha}(B_{\frac12}^+\cup B'_{\frac12})}\le C\norm u_{L^2(B_1)},$$ for some $C>0$.
		\begin{proof}
			For the classical proof see \cite{caf79}. For an alternative approach we refer to \cite{psu12}.
		\end{proof}
	\end{theorem}
	\subsection{Monotonicity formulas and blow-up}
	
	The following proposition claim the monotonicity formula of Almgren's frequency function.
	\begin{proposition}\label{a}
		Let $u\in H^1(B_1)$ a solution, $x_0\in\Gamma(u)$ and \begin{equation*}N^{x_0}(r,u):=\frac{r\int_{B_r(x_0)} \lvert \nabla u \vert^2\,dx}{\int_{\partial B_r(x_0)} u ^2\,d\mathcal{H}^n}
		\end{equation*} the Almgren's frequency function (dropping the dependence on $x_0$ if $x_0=0$), then $$r\mapsto N^{x_0}(r,u)$$ is monotone increasing for $ r\in(0,1-\lvert x_0\rvert).$
		
		Moreover $N(r,u)\equiv \lambda$ for all $r\in(0,1)$ if and only if $u$ $\lambda-$homogenous.
		\begin{proof} 
		For the proof we refer to \cite{acs08} or \cite{psu12}.
		\end{proof}
	\end{proposition}
	
	Now another fundamental monotonicity formula.
	\begin{lemma}\label{lemma2cose}
		Let $u\in H^1(B_1)$ a solution and $x_0\in \Gamma(u)$, we define $$H^{x_0}(r,u):=\int_{\partial B_r(x_0)}u^2\,d\mathcal{H}^n,$$ then the function
		$$r\mapsto \frac{H^{x_0}(r,u)}{r^{n+2\lambda}}$$ is monotone increasing for $r\in(0,1-\lvert x_0\rvert),$ where $\lambda= N^{x_0}(0^+,u)$.
		\begin{proof} For the proof see \cite{acs08} or \cite{psu12}. We only sketch the main idea for the sake of completeness.
			
			Let $H(r)=H^{x_0}(r,u)$ and $D(r)=\int_{B_r(x_0)}\lvert\nabla u \rvert^2 \,dx$, since $$H'(r)=\frac nrH(r)+2D(r),$$ then $$r\frac{H'(r)}{H(r)}=n+2\frac{rD(r)}{H(r)}\ge n+2\lambda,$$  by monotonicity of Almgren's frequency.
			
			Integrating this inequality we obtain the claim.
		\end{proof}
	\end{lemma}
	Now the proposition that shows the existence of a blow-up and the estimate $N^{x_0}(0^+,u)>1$ for all $x_0\in\Gamma(u)$.
	\begin{proposition}\label{blow}
		Let $u\in H^1(B_1)\cap C^{1,\alpha}_{loc}(B_1^+)$ a non-trivial solution and let $\{u_r\}_{r>0}$ the rescaled of $u$, defined as \begin{equation*}\label{rescaled}
			u_{r}(x):=\frac{u(x_0+rx)}{\left(\frac{1}{r^n}\int_{\partial B_r(x_0)}u^2\,d\mathcal{H}^n\right)^\frac12},
		\end{equation*} then there is $u_0\in H^1_{loc}(\mathbb{R}^{n+1})\cap C^{1,\alpha}_{loc}(\overline{\mathbb{R}^{n+1}_+})$, up to reducing $\alpha$, which is a non-trivial $\lambda-$homogeneous global solution, such that $u_{r}\to u_0$ in $C^{1,\alpha}_{loc}(\overline { \mathbb{R}_+^{n+1}})$, up to subsequences, with $\lambda=N^{x_0}(0^+,u)$. 
		
		Moreover, if $x_0\in\Gamma(u)$, then $N^{x_0}(0^+,u)\ge1+\alpha>1$.
		\begin{proof} 
		For the sake of completeness, we will give the idea of the proof, for the first part, and the complete proof of the second part. For more details see \cite{psu12} or \cite{survey}.
			
			Since for all $R>0$ $\lVert u_r\rVert_{L^2(\partial B_R)}\le C(R),$ and $\lVert \nabla u_r\rVert_{L^2(B_R)}\le C(R),$ then, by Poincaré inequality, $\lVert u_r\rVert_{L^2(B_R)}\le C(R).$ 
			
			Therefore for all $ R>0$, we have $$\lVert u_r\rVert _{C^{1,\alpha}(B_{\frac R2}^+\cup B'_{\frac R2})}\le C\lVert u_r\rVert_{L^2(B_R)}\le C(R),$$ by Theorem \ref{c11a}. Thus, up to reducing $\alpha$, we deduce the convergence in $C^{1,\alpha}_{loc}(\overline { \mathbb{R}_+^{n+1}})$, up to subsequences, to some function $u_0\in C^{1,\alpha}_{loc}(\overline { \mathbb{R}_+^{n+1}})$, that is a global solution by this convergence.
			
			Moreover $$N(r,u_0)=\lim_{r_k\to0} N(r,u_{r_k})=\lim_{r_k\to0}N^{x_0}(r_kr,u)=N^{x_0}(0^+,u)=\lambda,$$ then $u_0$ is $\lambda-$homogeneous.
			
			Now let $x_0\in\Gamma(u)$, then $ u(x_0)=\lvert \nabla u (x_0)\rvert=0$, using the Signorini ambiguous conditions and $C^{1,\alpha}$ regularity, thus $u_0(0)=\lvert\nabla u_0 (0)\rvert=0$, by $C^{1,\alpha}$ convergence.
			
			Moreover, since $u_0$ is non-trivial and homogenenous, we can choose, wlog, $x\in B_1^+$ such that $u_0(x)\not=0$, then, by $C^{1,\alpha} $ regularity, we have
			\begin{equation*}
				\lvert \nabla u_0(y)\rvert \le C\lvert y \rvert^\alpha \quad \forall y\in B_1^+ 
			\end{equation*}
			for some $C>0$. Therefore for all $ y\in B_1^+$
			$$\frac {u_0(y)}{\lvert y\rvert^{1+\alpha}}=\frac {1}{\lvert y\rvert^{1+\alpha}}\int_0^1 \nabla u_0(ty)\cdot y \,dt\le\frac{1}{\lvert y\rvert^{\alpha}}\int_0^1 \lvert \nabla u_0(ty) \rvert \,dt \le C,$$ thus for all $ t\in(0,1)$ $$t^{\lambda-(1+\alpha)}\frac{u_0(x)}{\lvert x\rvert^{1+\alpha}}=\frac { u_0(t x)}{\lvert tx\rvert^{1+\alpha}}\le C $$
			i.e. $\lambda=N^{x_0}(0^+,u)\ge1+\alpha>1$, as $t\to0^+$.
		\end{proof}
	\end{proposition}
	
	\subsection{Eigenfunctions of spherical Laplacian}
	Now we proceed with a brief reminder of the spherical Laplacian with its eigenfunctions and eigenvalues. 
	\begin{remark}\label{spherical}
		Let $\Delta_S$ the spherical Laplacian on $\partial B_1\subset\mathbb{R}^{n+1}$, then there are
		$\{\lambda_k\}_{k\in\mathbb{N}}\subset \mathbb{R}_{\ge0}$ increasing and $\{\phi_k\}_{k\in\mathbb{N}}\subset H^1(\partial B_1)$ normalized in $L^2(\partial B_1)$, such that $$-\Delta_S \phi_k=\lambda_k \phi_k,$$ with $\{\phi_k\}_{k\in\mathbb{N}}$ orthonormal basis of $H^1(\partial B_1)$.
		
		Moreover, in spherical coordinates, \begin{equation}\label{sphh}\Delta=\frac{\partial^2}{\partial r^2}+\frac{n}{r}\frac{\partial}{\partial r}+\frac{1}{ r^2} \Delta_S,
		\end{equation} then $r^{\alpha}\phi(\theta)$ is harmonic in $\mathbb{R}^{n+1}$ if and only if $\phi$ is an eigenfunction of eigenvalue $\lambda(\alpha):=\alpha(\alpha+n-1)$. In this case $r^{\alpha}\phi(\theta)$ is a polynomial, i.e. $\alpha\in\mathbb{N}$, by the gradient estimates for harmonic functions.
		
		Furthermore, for all $  \lambda\ge0$, we define $$E(\lambda)=\{\phi\in H^1(\partial B_1): -\Delta_S \phi=\lambda \phi, \ \lVert \phi\rVert_{L^2({\partial B_1})}=1\}.$$ the (normalized) eigenspace of eigenvalue $\lambda$, then
		\begin{enumerate}
			\item For $\alpha=0$, then $\lambda_1=0$ and the eigenspace coincides with the space of constant functions.
			
			\item For $\alpha=1$, then $\lambda_2=\ldots=\lambda_{n+2}=n$ and the eigenspace (of dimension $n+1$) coincides with the space of linear functions.
			
			\item For $\alpha=2$, then $\lambda_{k}\ge2(n+1)=\lambda(2)$ for all $ k\ge n+3.$
		\end{enumerate}
		
	\end{remark}
	The following lemma, from \cite{csv17}, will be used in the proof of the epiperimetric inequality in Theorem \ref{thm1} and to show the frequency gap in Proposition \eqref{gap}.
	\begin{lemma} \label{spherical1} Let $\phi\in H^1(\partial B_1)$ with $$\phi(\theta)=\sum_{k=1}^{\infty} c_k \phi_k(\theta)\in H^1(\partial B_1),$$ where $\phi_k$ normalized eigenfunctions of spherical Laplacian as above, and let $r^\frac32\phi(\theta)$, the $\frac32-$homogeneous extension, then 
		\begin{equation}\label{prima}
			W_\frac32(r^\frac32 \phi)=\frac{1}{n+2}\sum_{k=1}^{\infty}\left(\lambda_k-\lambda\left(\frac32\right)\right)c_k^2.
		\end{equation}
		Moreover, if $c\in H^1(\partial B_1)$ such that $r^{\frac32+t}c$ is a solution, then 
		\begin{equation}\label{terza}
			W_\frac32(r^{\frac32+t}c)=t\norm c_{L^2(\partial B_1)}^2
		\end{equation}
		and
		\begin{equation}\label{quarta}
			W_\frac32(r^{\frac32}c)=\left(1+\frac{t}{n+2}\right)W_\frac32(r^{\frac32+t}c).
		\end{equation}
		\begin{proof}
			The proof is in \cite{csv17}, but we report for the sake of completeness. 
			
			Since for a function $w$ that is $ \lambda-$homogeneous holds \begin{equation}\label{www}\int_{B_1}\lvert \nabla w\rvert^2 \,dx=\frac{1}{n+2\lambda-1}\left(\int_{\partial B_1}\lvert \nabla_\theta w\rvert^2\,d\mathcal{H}^n+\lambda^2\int_{\partial B_1}w^2\,d\mathcal{H}^n\right),
			\end{equation} then 	$$\begin{aligned}
				&W_\frac32(r^\frac32\phi)=\frac{1}{n+2}\left(\int_{\partial B_1}\lvert\nabla_\theta \phi\rvert^2 \,dx+\frac94\int_{\partial B_1}\phi^2\,d\mathcal{H}^n\right)-\frac32\int_{\partial B_1}\phi^2\,d\mathcal{H}^n=\\&=\sum_{k=1}^\infty c_k^2\left(\frac{\lambda_k+\frac94}{n+2}-\frac32\right)=\frac{1}{n+2}\sum_{k=1}^{\infty}\left(\lambda_k-\lambda\left(\frac32\right)\right)c_k^2,
			\end{aligned}$$ that is\eqref{prima}
			Moreover, if $r^{\frac32+t}c$ is a solution, then $W_{\frac32+t}(r^{\frac32+t}c)=0$, therefore \begin{equation*}\label{}
				\begin{aligned} 
					W_\frac32(r^{\frac32+t}c)&=\lVert \nabla (r^{\frac32+t}c)\rVert_{L^2(B_1)}^2-\frac32\lVert c \rVert_{L^2(\partial B_1)}^2=\\&=W_{\frac32+t}(r^{\frac32+t}c)+t\lVert c \rVert_{L^2(\partial B_1)}^2=t\lVert c \rVert_{L^2(\partial B_1)}^2,
				\end{aligned}
			\end{equation*}
			that is \eqref{terza}.
			Finally, since
			\begin{equation}\label{yr}
				\lVert \nabla_\theta c\rVert_{L^2(\partial B_1)}^2=\lambda\left(\frac32+t\right)\lVert c \rVert_{L^2(\partial B_1)}^2,
			\end{equation} then, using \eqref{www}, we get
			\begin{equation*}\label{}
				\begin{aligned} 
					W_\frac32(r^{\frac32}c)&=\frac{1}{n+2}\left(\lVert\nabla_\theta c \rVert_{L^2(\partial B_1)}^2+\frac94\lVert c \rVert_{L^2(\partial B_1)}^2\right)-\frac32\lVert c \rVert_{L^2(\partial B_1)}^2=\\&=\frac{1}{n+2}\left(\lVert\nabla_\theta c \rVert_{L^2(\partial B_1)}^2-\lambda\left(\frac32\right)\lVert c \rVert_{L^2(\partial B_1)}^2\right)=\\&=\frac{\lambda(\frac32+t)-\lambda(\frac32)}{n+2}\lVert c \rVert_{L^2(\partial B_1)}^2=\left(1+\frac{t}{n+2} \right)t\lVert c \rVert_{L^2(\partial B_1)}^2=\\&=\left(1+\frac{t}{n+2} \right)W_\frac32(r^{\frac32+t}c),
				\end{aligned}
			\end{equation*}
			where in the last equality we have used \eqref{terza}, then we have obtained \eqref{quarta}.
		\end{proof}
	\end{lemma}
	
	\subsection{Properties of $h_e$}
	Finally we recall the properties of the function $h_e$ defined in \eqref{he}, which is the only $\frac32-$homogeneous solution. 

	Notice that the latter is often demonstrated together with optimal regularity, in fact we will not use this fact, but we will prove in Proposition \ref{uniq}.
	\begin{proposition} \label{heprop} Let $h_e$ as in \eqref{he}, then
		$h_e$ is a $\frac32-$homogenenous solution of the thin obstacle problem, $h_e=0$ on $B'_1\cap\{x\cdot e\le0\}$ and the derivative in $x_{n+1}$ direction is
		\begin{equation}\label{derivhe}
			\partial_{n+1}h_e=
			\begin{cases} 
				-\frac32\lvert x'\cdot e \rvert^\frac12 & B'_1\cap\{x'\cdot e <0\}\\ 
				0 & B'_1\cap\{x'\cdot e \ge0\}.
			\end{cases}
		\end{equation}
		Moreover, for all $ \eta\in H^1(B_1)$ even across $\{x_{n+1}=0\}$, we have
		\begin{equation}\label{prophe}
			\int_{B_1} -\Delta h_e\eta\,dx=3\int_{B'_1\cap\{x'\cdot e<0\}}\eta(x',0) \lvert x'\cdot e\rvert^\frac12\,d\mathcal{H}^n,
		\end{equation} by the Euler-Lagrange equations.
		
		In particular, if $\eta\in H^1(B_1)$ such that $\eta=0$ on $B'_1\cap\{(x'\cdot e)<0\}$, then
		\begin{equation}\label{prophe1}
			\int_{B_1} \nabla h_e\cdot \nabla \eta\,dx=\frac32\int_{\partial B_1}h_e\eta\,d\mathcal{H}^n
		\end{equation}
		and $W_\frac32 (h_e)=0$, choosing $\eta=h_e$.
		
		Finally the $L^2(\partial B_1)$ projection of $h_e$ on linear functions has the form $C(x'\cdot e)$ for some $C>0$.
		\begin{proof} We only sketch the proof and we refer to \cite{fs16} for more details.
			With a slight abuse of notation, $h_e$ is $2-$dimensional, i.e. $h_e(x)=h_e(x'\cdot e,x_{n+1})$, then we can calculate explicitly $\partial_{n+1} h_e$ in polar coordinates, to deduce \eqref{derivhe}. 
			
			Moreover \eqref{prophe} is the Laplacian in measure sense for solutions and \eqref{prophe1} follows by $x\cdot \nabla=\partial _{\nu}$ on $\partial B_1$ and an integration by parts.
			
			Finally, for $L^2(\partial B_1)$ projection of $h_e$, let's consider again $h_e$ as $2-$dimensional. An explicit calculation, after a linear change of variables with a basis that includes $e$, gives the claim.
		\end{proof}
	\end{proposition}
	\section{Epiperimetric inequality for negative energies $W_\frac32$}
	In this section we will prove Theorem \ref{thm1}. The strategy is to decompose a function $c\in H^1(\partial B_1)$ using eigenfunctions of spherical Laplacian, similar to the decomposition in \cite{csv17}. 
	
	Let $c\in H^1(\partial B_1)$, even with respect to the hyperplane $\{x_{n+1}=0\}$, then
	since $E(\lambda_2)$ is the space of linear functions, the projection in $L^2(\partial B_1)$ of $c$ has the form $c_1 (x'\cdot e)$, for some $e\in \partial B'_1$ and $c_1\ge0$.
	
	Choosing $e\in\partial B'_1$ as above, even the projection of $h_e$ in $E(\lambda_2)$ has the form $C(x'\cdot e)$, for $C>0$, by Proposition \ref{heprop}, then we can choose $C\ge0$ such that $Ch_e$ and $c$ have the same projection on $E(\lambda_2)$.
	
	Let $u_0: \partial B_1\to\mathbb{R}$ such that $u_0(\theta)=\lvert \theta_{n+1}\rvert$, therefore is even and thus orthogonal to $E(\lambda_2)$. Since $E(\lambda_1)$ is the space of constant functions, we can find $c_0\in\mathbb{R}$ such that $c_0u_0$ and $c-Ch_e$ have the same projection on $E(\lambda_1)$. Hence we can decompose $c$ as 
	\begin{equation}\label{decomposition}
		c(\theta)=Ch_e(\theta)+c_0u_0(\theta)+\phi(\theta),
	\end{equation}
	with $C\ge0$ and $$\phi=\sum_{\{k:\ \lambda_k\ge\lambda(2)\}} c_k\phi_k,$$ where $\phi_k$ are the normalized eigenfunctions as above.
	
	In the following, with a slight abuse of notations, we denote with $h_e$ the function in $\partial B_1$, i.e. $h_e(\theta).$
	
	\begin{proof}[Proof of Theorem \ref{thm1}] 
			Let $z$ a $\frac32-$homogeneous extension of its trace $c\in H^1(\partial B_1)$, then if decompose $z$ as $$z(r,\theta)=Cr^\frac32h_e(\theta)+c_0r^\frac32u_0(\theta)+r^\frac32\phi(\theta), $$ therefore the explicit competitor is $$\zeta(r,\theta)=Cr^\frac32h_e(\theta)+c_0ru_0(\theta)+r^\frac32\phi(\theta).$$
		
			First notice that $\zeta=c$ in $\partial B_1$, by \eqref{decomposition}. Moreover, since $u_0=0$ on $B'_1$, then $$\zeta=Cr^\frac32h_e(\theta)+c_0ru_0(\theta)+r^\frac32\phi(\theta)=
			Cr^\frac32h_e(\theta)+c_0r^\frac32u_0(\theta)+r^\frac32\phi=z\ge0$$ on $B'_1$, which is $\zeta\in\mathcal{K}_c$.
			
			Now we want to compute the Weiss' energy of $Cr^\frac32h_e(\theta)+c_0r^\alpha u_0(\theta)+r^\frac32\phi(\theta)$, for $\alpha=1,\frac32$. 
			By Proposition \ref{heprop}, we have $W_\frac32(h_e)=0$, therefore
			\begingroup
			\allowdisplaybreaks
			\begin{align*}
				&W_\frac32(Cr^\frac32h_e+c_0r^\alpha u_0+r^\frac32\phi)=C^2W_\frac32(h_e)+W_\frac32(c_0r^\alpha u_0+r^\frac32\phi)+\\&\qquad+2C\Biggl(\int_{B_1}\nabla (r^\frac32h_e)\cdot \nabla(c_0r^\alpha u_0+r^\frac32\phi)\,dx+\\&\qquad-\frac32\int_{\partial B_1}h_e(c_0u_0+\phi)\,d\mathcal{H}^n\Biggl) =\\&=c_0^2W_\frac32(r^\alpha u_0)+W_\frac32(r^\frac32\phi)+\\&\qquad+2c_0\Biggl( \int_{B_1}\nabla (r^\alpha u_0)\cdot \nabla (r^\frac32\phi)\,dx-\frac32\int_{\partial B_1}u_0\phi\,d\mathcal{H}^n\Biggl)+\\&\qquad+2C\Biggl( \int_{B_1}\nabla (r^\frac32h_e)\cdot \nabla (r^\frac32\phi)\,dx-\frac32\int_{\partial B_1}h_e\phi\,d\mathcal{H}^n\Biggl),
			\end{align*}%
			\endgroup where in the last equality we have used that $u_0\equiv$ 0 on $B'_1$, combined with \eqref{prophe1}.
			
			Since $x\cdot \nabla=\partial _{\nu}$ on $\partial B_1$ and using an integration by parts, it follows that  $$\begin{aligned}
				&W_\frac32(Cr^\frac32h_e+c_0r^\alpha u_0+r^\frac32\phi)=c_0^2W_\frac32(r^\alpha u_0)+W_\frac32(r^\frac32\phi)+\\&\qquad+2c_0\Biggl( \int_{B_1}\nabla (r^\alpha u_0)\cdot \nabla (r^\frac32\phi)\,dx-\frac32\int_{\partial B_1}u_0\phi\,d\mathcal{H}^n\Biggl)+\\&\qquad+2C\Biggl(\int_{B_1} -\Delta (r^\frac32h_e) r^\frac32\phi \,dx\Biggl),
			\end{aligned}$$
			then
			
			\begin{equation}\label{ijkl}W_{\frac{3}{2}}(\zeta)- (1+\varepsilon)W_{\frac{3}{2}}(z)=I+J+K+L,
			\end{equation} where $$I=c_0^2\left(W_\frac32(r u_0)-(1+\varepsilon)W_\frac32(r^\frac32 u_0)\right),$$ $$J=W_\frac32(r^\frac32\phi)-(1+\varepsilon)W_\frac32(r^\frac32\phi),$$ $$\begin{aligned}K&=2c_0\Biggl( \int_{B_1}\nabla (r u_0)\cdot \nabla (r^\frac32\phi)\,dx-\frac32\int_{\partial B_1}u_0\phi\,d\mathcal{H}^n+\\&\qquad-(1+\varepsilon)\Biggl(\int_{B_1}\nabla (r^\frac32 u_0)\cdot \nabla (r^\frac32\phi)\,dx-\frac32\int_{\partial B_1}u_0\phi\,d\mathcal{H}^n\Biggl)\Biggl)
			\end{aligned}$$ and $$\begin{aligned}L&=2C\Biggl(\int_{B_1} -\Delta (r^\frac32h_e) r^\frac32\phi \,dx-(1+\varepsilon)\Biggl(\int_{B_1} -\Delta (r^\frac32h_e) r^\frac32\phi \,dx\Biggl)\Biggl).
			\end{aligned}$$
			
			For $I$, we notice that the function $-ru_0(\theta)=-\lvert x_{n+1} \rvert $ is a solution, then using \eqref{terza}, we obtain $$W_\frac32(ru_0)=W_\frac32(-r^{\frac32-\frac12}u_0)=-\frac12\lVert u_0\rVert_{L^2(\partial B_1)}^2$$ and using \eqref{quarta}, we get $$W_\frac32(r^\frac32u_0)=\left(1+\frac{-\frac12}{n+2}\right)W_\frac32(ru_0)=\left(1+\frac{-\frac12}{n+2}\right)\left(-\frac12\right)\lVert u_0\rVert_{L^2(\partial B_1)}^2,$$ therefore $$I=\left(-\frac12-\left(1+\varepsilon\right)\left(1+\frac{-\frac12}{n+2}\right)\left(-\frac12\right)\right)\lVert u_0\rVert_{L^2(\partial B_1)}^2=0,$$ since $\varepsilon =\frac1{2n+3}$, with a simple calculation.
			
			For $J$, using \eqref{prima}, we deduce that $$J=-\varepsilon W_\frac32(r^\frac32\phi)=-\frac\varepsilon{n+2}\sum_{k:\ \lambda_k\ge \lambda(2)}\left(\lambda_k-\lambda\left(\frac32\right)\right)c_k^2\le0,$$ since $\lambda(2)\ge\lambda(\frac32)$.
			
			For $K$, since $x\cdot \nabla=\partial _{\nu}$ on $\partial B_1$, then
			$$\begin{aligned}K&=2c_0\Biggl(\int_{B_1}\nabla (r u_0)\cdot \nabla (r^\frac32\phi)\,dx-\int_{\partial B_1}u_0\phi\,d\mathcal{H}^n-\frac12\int_{\partial B_1}u_0\phi\,d\mathcal{H}^n+\\&\qquad-(1+\varepsilon)\Biggl(\int_{B_1}\nabla (r^\frac32 u_0)\cdot \nabla (r^\frac32\phi)\,dx-\frac32\int_{\partial B_1}u_0\phi\,d\mathcal{H}^n\Biggl)\Biggl)
				=\\&=2c_0\Biggl( \int_{B_1}-\Delta (r u_0)r^\frac32\phi\,dx-\frac12\int_{\partial B_1}u_0\phi\,d\mathcal{H}^n+\\&\qquad-(1+\varepsilon)\Biggl(\int_{B_1}-\Delta (r^\frac32 u_0)r^\frac32\phi\,dx\Biggl)\Biggl),\end{aligned}$$ where we have used the integration by parts.
			
			Now, by \eqref{sphh}, we get $$\Delta (r^\alpha u_0)=\lambda(\alpha)r^{\alpha-2}u_0+r^{\alpha-2}\Delta_S u_0,$$ with $\lambda(1)=n$ and $\lambda(\frac32)=\frac32n+\frac34$, then
			$$\begin{aligned}K&=2c_0\Biggl( \int_{B_1}-nr^{-1}u_0r^\frac32\phi\,dx-\int_{B_1}r^{-1}\Delta_S u_0r^\frac32\phi\,dx-\frac12\int_{\partial B_1}u_0\phi\,d\mathcal{H}^n+\\&\qquad-(1+\varepsilon)\Biggl(\int_{B_1}-\left(\frac32n+\frac34\right)r^{-\frac12}u_0r^\frac32\phi\,dx-\int_{B_1}r^{-\frac12}\Delta_S u_0r^\frac32\phi\,dx\Biggl)\Biggl)=\\&=2c_0\Biggl( -\frac{n}{n+\frac32}\int_{\partial B_1}u_0\phi\,d\mathcal{H}^n-\frac{1}{n+\frac32}\int_{\partial B_1}\Delta_S u_0\phi\,d\mathcal{H}^n+\\&\qquad-\frac12\int_{\partial B_1}u_0\phi\,d\mathcal{H}^n-(1+\varepsilon)\Biggl(-\left(\frac32n+\frac34\right)\left(\frac{1}{n+2}\right)\int_{\partial B_1}u_0\phi\,d\mathcal{H}^n+\\&\qquad-\frac{1}{n+2}\int_{\partial B_1}\Delta_S u_0\phi\,d\mathcal{H}^n\Biggl)\Biggl),\end{aligned}$$ where in the last equality we have used that if $w$ is $\lambda-$homogeneous, then $$\int_{B_1} w\,dx=\frac1{n+\lambda+1}\int_{\partial B_1} w\,d\mathcal{H}^n,$$ with a simple change of variables. 
			
			Hence, we obtain
			$$\begin{aligned}K&=2c_0\Biggl( -\frac{n}{n+\frac32}-\frac12+\left(1+\varepsilon\right)\left(\frac32n+\frac34\right)\left(\frac{1}{n+2}\right)\Biggl) \int_{\partial B_1}u_0\phi\,d\mathcal{H}^n +\\&+2c_0\Biggl(-\frac{1}{n+\frac32}+\left(1+\varepsilon\right)\left(\frac{1}{n+2}\right)\Biggl)\int_{\partial B_1}\Delta_S u_0\phi\,d\mathcal{H}^n=0,
			\end{aligned}$$ since $\varepsilon =\frac{1}{2n+3}$, with a simple calculation. 
			
			For $L$, by \eqref{prophe}, we have $$L=-2C\varepsilon\left( \int_{B_1} -\Delta(r^\frac32h_e)r^\frac32 \phi\,dx\right)=-6C\varepsilon\int_{B'_1\cap\{x'\cdot e<0\}}\phi \lvert x'\cdot e\rvert^\frac12\,d\mathcal{H}^n\le0,$$ since $h_e(\theta)=u_0(\theta)=0$ for $\theta\in B'_1\cap\{x'\cdot e<0\}$, then $\phi=c\ge0$ in $B'_1\cap\{x'\cdot e<0\}$.
			
			Finally, since $I,J,K,L\le0$, we conclude using \eqref{ijkl}.
		\end{proof}
	\begin{remark}\label{0} This epiperimetric inequality is an equality if $I=J=K=L=0$. In particular, by $J=0$, we deduce $\phi\equiv0$, that is $$c=Ch_e(\theta)+c_0u_0(\theta),$$ for some $C\ge0$.
	\end{remark}
	
	\section{Frequency gap} By an epiperimetric inequality for negative energies $W_{\frac32}$, we can deduce a backward frequency gap for the frequency $\frac32$. In particular, the constant $\varepsilon=\frac{1}{2n+3}$ in Theorem \ref{thm1} is the best possible constant, since we can show the backward frequency gap until 1, that is Proposition \ref{gap}. 
	
	\begin{proof}[Proof of Proposition \ref{gap}] Let $c\in H^1(\partial B_1)$ a trace of a $(\frac32+t)-$homogeneous global solution with $t<0$, say $r^{\frac32+t}c(\theta)$. Therefore, for the first part, it is sufficent to check that $t\le -\frac12$.
			
			Since $t<0$, then $W_{\frac32}(r^{\frac32+t}c)=t\norm c_{L^2(\partial B_1)}^2<0 $ by \eqref{terza}, thus, using the epiperimetric inequality for negative energies, i.e. Theorem \ref{thm1}, we deduce that $$\begin{aligned}W_{\frac32}(r^{\frac32+t}c)&\le W_{\frac32}(\zeta)\le (1+\varepsilon)W_{\frac32}(r^{\frac32}c)=\\&=\left(1+\frac1{2n+3} \right)\left(1+\frac{t}{n+2}\right)W_{\frac32}(r^{\frac32+t}c),
			\end{aligned}$$ where we have used \eqref{quarta}.
			
			Hence, since we have a negative energies, we get $$\left(1+\frac1{2n+3}\right)\left(1+\frac{t}{n+2}\right)\le 1, $$ which is, after a simple computation, $t\le- \frac12$, which is what we wanted to prove.
			 
	Now the estimate \eqref{f} becomes trivial.
	In fact, let $x_0\in\Gamma(u)$ and $u_0$ the blow-up of $u$ around $x_0$ defined as above, then by Proposition \ref{blow}, $u_0$ is a global solution $\lambda-$homogeneous, with $\lambda=N^{x_0}(0^+,u)\ge1+\alpha>1$. 
			
			But $u_0$ cannot be $\lambda-$homogenenous with $\lambda\in (1,\frac32)$, then we have $\lambda\ge\frac32$, which is \eqref{f}.
		\end{proof}	
	\section{Optimal regularity}
	For the sake of completeness, we show how to use the estimate \eqref{f} for the proof of optimal $C^{1,\frac12} $ regularity.
	
	By Lemma \ref{lemma2cose} we obtain that if $K\subset B_1$ is a compact set, then for all $ r\in(0,1-\lvert x_0\rvert)$
	$$H^{x_0}(r)\le C \norm u _{L^2(B_1)}^2r^{n+2\lambda}\le C \norm u _{L^2(B_1)}^2r^{n+3} \quad \forall x_0\in \Gamma(u)\cap K,$$ since $\lambda\ge\frac32$,
	thus for all $ r\in(0,1-\lvert x_0\rvert)$ \begin{equation}\label{funda}
		\norm u _{L^2(B_r(x_0))}\le C \norm u _{L^2(B_1)}r^{\frac n2+2} \quad \forall x_0\in \Gamma(u)\cap K,
	\end{equation}since $\norm u _{L^2(B_r(x_0))}^2=\int_0^rH^{x_0}(s)\,ds,$
	
	With this estimate, we can prove the $C^{1,\frac12}$ regularity. Roughly speaking, we can use the gradient estimates for harmonic functions far from the free boundary, and we the estimate \eqref{funda} near the free boundary.
	\begin{theorem} \label{c112}($C^{1,\frac12}$ estimate)
		Let $u\in H^1(B_1)$ a solution of the thin ostacle problem, then $u\in C^{1,\frac12}_{loc}(B_1^+\cup B'_1)$ with the estimate $$\norm u_{C^{1,\frac12}(B_{\frac12}^+\cup B'_{\frac12})}\le C\norm u_{L^2(B_1)},$$ for some $C>0$
		\begin{proof} We define $d(x)=\mbox{dist}(x,\Gamma(u))\le2$, for $x\in B_1$.
			
			Let $x,y\in B_\frac12^+$ such that $d(x)=\lvert x-x_0\rvert$ and $d(y)=\lvert y-y_0\rvert$, for some $x_0,y_0\in\Gamma(u)$. We can suppose wlog $\lvert x-y\rvert\le\frac1{64}$ and $d(x)\ge d(y)$. 
			
			Notice that 
			$u$ or the odd symmetric of $u$ respect to $\{x_{n+1}=0\}$ is harmonic in $B_{d(x)}(x)\cap B_1$. With a slight abuse of notation, we denote with $u$ the function with appropriate symmetry.
			We have 2 cases:
			
			1. If $d(x)\ge\frac18$, then 
			$\lvert x-y\rvert\le\frac1{64}\le\frac{d(x)}{8},$
			i.e. $y\in B_{\frac{d(x)}{8}}(x)$. 
			Follows that, using the gradient estimates for harmonic functions, \begin{equation*}\begin{aligned} \lVert D^2u\rVert_{L^{\infty}(B_{\frac{d(x)}{8}}(x))}&\le\frac{C}{d(x)^{\frac{n+1}{2}+2}}\lVert u\rVert_{L^{2}(B_{\frac{d(x)}{4}}(x))}\le C\norm u_{L^2(B_1)},
			\end{aligned}\end{equation*}
			where we have used $B_{\frac{d(x)}{4}}(x)\subset B_\frac12(x)\subset B_1$.
			Therefore $$\begin{aligned} \lvert \nabla u(x)-\nabla u (y)\rvert&\le \lVert D^2u\rVert_{L^{\infty}(B_{\frac{d(x)}{8}}(x))}\lvert x-y\rvert\le C\norm u_{L^2(B_1)}\lvert x-y\rvert\le\\&\le C\norm u_{L^2(B_1)}\lvert x-y\rvert^\frac12.
			\end{aligned}$$

			2. If $d(x)<\frac18$, then
			$x_0\in B_{\frac58}$, hence we can use \eqref{funda} with $K=\overline{B_\frac58}$. By the gradient estimates for harmonic functions and \eqref{funda}, we obtain 
			\begin{equation*}\begin{aligned}\label{due} \lVert D^2u\rVert_{L^{\infty}(B_{\frac{d(x)}{2}}(x))}&\le\frac{C}{d(x)^{\frac{n+1}{2}+2}}\lVert u\rVert_{L^{2}(B_{d(x)}(x))}\le\frac{C}{d(x)^{\frac{n+1}{2}+2}}\lVert u\rVert_{L^{2}(B_{2d(x)}(x_0))}\le\\&\le
					\frac{C}{d(x)^{\frac{n+1}{2}+2}}\norm u_{L^2(B_1)}d(x)^{\frac n2+2}=C\norm u_{L^2(B_1)}d(x)^{-\frac12}.
			\end{aligned}\end{equation*} 
			where we have used that $B_{2d(x)}(x_0)\subset B_{\frac14}(x_0)\subset B_1$, since $x_0\in B_{\frac58}$.
			
			Similarly
			\begin{equation*}\begin{aligned}\label{uno} \lvert \nabla u(x)\rvert&\le\frac{C}{d(x)^{\frac{n+1}{2}+1}}\lVert u\rVert_{L^{2}(B_{d(x)}(x))}\le\frac{C}{d(x)^{\frac{n+1}{2}+1}}\lVert u\rVert_{L^{2}(B_{2d(x)}(x_0))}\le\\&\le
					\frac{C}{d(x)^{\frac{n+1}{2}+1}}\norm u_{L^2(B_1)}d(x)^{\frac n2+2}=C\norm u_{L^2(B_1)}d(x)^{\frac12},
			\end{aligned}\end{equation*} and the same for $y$.

			Now suppose that $\lvert x-y\rvert\le\frac{d(x)}{2}$, then $y\in B_{\frac{d(x)}{2}}(x)$, therefore we deduce $$\begin{aligned} \lvert \nabla u(x)-\nabla u (y)\rvert&\le \lVert D^2u\rVert_{L^{\infty}(B_{\frac{d(x)}{2}}(x))}\lvert x-y\rvert\le C\norm u_{L^2(B_1)}d(x)^{-\frac12}\lvert x-y\rvert\le\\&\le C\norm u_{L^2(B_1)}\lvert x-y\rvert^\frac12,
			\end{aligned}$$
			while, if $\lvert x-y\rvert>\frac{d(x)}{2}\ge\frac{d(y)}{2}$, then $$\begin{aligned} \lvert \nabla u(x)-\nabla u (y)\rvert&\le \lvert\nabla u(x)\rvert+\lvert\nabla u(y)\rvert\le C\norm u_{L^2(B_1)}(d(x)^{\frac12}+d(y)^\frac12)\le\\&\le C\norm u_{L^2(B_1)}\lvert x-y\rvert^\frac12,
			\end{aligned}$$ that conclude the proof.
		\end{proof}
	\end{theorem}
\section{Characterization of $\frac32-$homogeneous solutions}
We conclude with another consequence of the epiperimetric inequality for negative energies, i.e. Theorem \ref{thm1}. Indeed, this inequality can be used to show the characterization of $\frac32-$homogeneous solutions, that are $Cr^\frac32h_e(\theta)$, for some $e\in\partial B_1'$ and $C\ge0$.
	\begin{proof}[Proof of Proposition \ref{uniq}]
		Let $z=r^\frac32 c(\theta)$, a $\frac32-$homogeneous global solution. Since the $\frac32-$Weiss' energy of a $\frac32-$homogeneous solution is 0, then $$0=W_\frac32(z)\le W_\frac32(\zeta)\le(1+\varepsilon)W_\frac32(z)=0,$$ i.e. the epiperimetric inequality is an equality. Therefore, by Remark \ref{0}, we deduce that $$z=Cr^\frac32h_e(\theta)+c_0r^\frac32u_0(\theta),$$ for some $C\ge0$, then it is sufficient to check that $c_0=0$.
		
		Since $u_0\equiv0$ on $B'_1$, using \eqref{prophe1}, we obtain \begin{equation}\label{final} 0=W_\frac32(z)=C^2W_\frac32(r^\frac32h_e)+c_0^2W_\frac32(r^\frac32u_0)=c_0^2W_\frac32(r^\frac32u_0),
			\end{equation} where we have used Proposition \ref{heprop} to have $W_\frac32(r^\frac32 h_e) =0$.
		
		Now, since $-ru_0$ is a solution, by \eqref{terza} and \eqref{quarta}, we deduce that $$W_\frac32(r^\frac32u_0)=\left(1+\frac{-\frac12}{n+2}\right)W_\frac32(ru_0)=\left(1+\frac{-\frac12}{n+2}\right)\left(-\frac12\right)\lVert u_0\rVert_{L^2(\partial B_1)}^2<0,$$ that is $c_0=0$, by \eqref{final}. 
	\end{proof}
	\bibliographystyle{alpha}
	\bibliography{paper1.bib}

\begin{thebibliography}{CSV17}

\bibitem[AC04]{ac04}
I.~Athanasopoulos and L.~Caffarelli.
\newblock Optimal regularity of lower dimensional obstacle problems.
\newblock {\em Zap. Nauchn. Sem. S.-Peterburg. Otdel. Mat. Inst. Steklov},
  2004.

\bibitem[ACS08]{acs08}
I.~Athanasopoulos, L.~Caffarelli, and S.~Salsa.
\newblock The structure of the free boundary for lower dimensional obstacle
  problems.
\newblock {\em Amer. J. Math 130}, 2008.

\bibitem[Alm00]{alm00}
F.~J. Almgren.
\newblock Almgren's big regularity paper: Q-valued functions minimizing
  dirichlet's integral and the regularity of area-minimizing rectifiable
  currents up to codimension 2.
\newblock {\em World Scientific Monograph Series in Mathematics}, 2000.

\bibitem[Caf79]{caf79}
L.~Caffarelli.
\newblock Further regularity for the {S}ignorini problem.
\newblock {\em Comm. Partial Diﬀerential Equations 4}, 1979.

\bibitem[CSV17]{csv17}
M.~Colombo, L.~Spolaor, and B.~Velichkov.
\newblock Direct epiperimetric inequalities for the thin obstacle problem and
  applications.
\newblock {\em Comm. Pure Appl. Math. 73}, 2017.

\bibitem[ESV23]{esv23}
N.~Edelen, L.~Spolaor, and B.~Velichkov.
\newblock The symmetric (log-)epiperimetric inequality and a decay-growth
  estimate.
\newblock 2023.

\bibitem[Fer20]{survey}
X~Fernandez{-Real}.
\newblock The thin obstacle problem: a survey.
\newblock 2020.

\bibitem[FS]{fs}
M.~Focardi and E.~Spadaro.
\newblock Thin obstacle problems: variational methods and geometric measure
  theory.
\newblock {\em In preparation}.

\bibitem[FS16]{fs16}
M.~Focardi and E.~Spadaro.
\newblock An epiperimetric inequality for the thin obstacle problem.
\newblock {\em Adv. Diﬀerential Equations, 21}, 2016.

\bibitem[GP09]{gp09}
N.~Garofalo and A.~Petrosyan.
\newblock Some new monotonicity formulas and the singular set in the lower
  dimensional obstacle problem.
\newblock {\em Invent. Math. 177}, 2009.

\bibitem[LS67]{ls67}
J.~L. Lions and G.~Stampacchia.
\newblock Variational inequalities.
\newblock {\em Comm. Pure Appl. Math. 20}, 1967.

\bibitem[PSU12]{psu12}
A.~Petrosyan, H.~Shahgholian, and N.~Uraltseva.
\newblock Regularity of free boundaries in obstacle-type problems.
\newblock {\em volume 136 of Graduate Studies in Mathematics. American
  Mathematical Society}, 2012.

\bibitem[Wei99]{wei99}
G.~S. Weiss.
\newblock A homogeneity improvement approach to the obstacle problem.
\newblock {\em Invent. Math. 138}, 1999.

\end{thebibliography}
\end{document}